\documentclass[11pt]{article}
\usepackage{amsmath}%
\usepackage[T1]{fontenc}
\usepackage{enumitem}
\usepackage{amsthm}
\usepackage{amsxtra}%
\usepackage{amsfonts}%
\usepackage{amssymb}%
\usepackage{graphicx}
\usepackage{subfigure}
\usepackage{color,hyperref}
\usepackage[margin=3cm]{geometry}
\usepackage{cite}

\newcommand{\iid}{\emph{i.i.d.}}
\newcommand{\A}{{\mathcal A}}
\newcommand{\B}{{\mathcal B}}
\renewcommand{\O}{{\mathcal O}}
\renewcommand{\bar}[1]{{\overline{#1}}}
\newcommand{\F}{{\mathcal F}}
\renewcommand{\P}{{\mathbb P}}

\newcommand{\R}{\mathbb{R}}
\newcommand{\N}{\mathbb{N}}
\newcommand{\E}{\mathbb{E}}
\newcommand{\M}{\mathbb{M}}
\newcommand{\Q}{\mathbb{Q}}

\newcommand{\Z}{\mathbb{Z}}
\newcommand{\vb}[1]{\mathbf{#1}}

\newcommand{\eps}{\varepsilon}
\renewcommand{\hat}[1]{\widehat{#1}}

\renewcommand{\phi}{\varphi}
\newcommand{\limsupn}{\limsup_{n\to \infty}}
\newcommand{\liminfn}{\liminf_{n\to \infty}}
\newcommand{\limn}{\lim_{n\to \infty}}
\newcommand{\ppp}[1]{\Pi_{#1}}
\renewcommand{\subset}{\subseteq}
\renewcommand{\supset}{\supseteq}

\newtheorem{theorem}{Theorem}
\newtheorem*{theorem*}{Theorem}
\theoremstyle{plain}

\newtheorem{lemma}{Lemma}

\numberwithin{equation}{section}

\title{A direct verification argument for the Hamilton-Jacobi equation continuum limit of nondominated sorting\thanks{The research described in this paper was partially supported by NSF grant DMS-1500829.}}
\author{Jeff Calder\thanks{Department of Mathematics, University of California, Berkeley. ({\tt jcalder@berkeley.edu})}}

\begin{document} 
\maketitle

\begin{abstract}
Nondominated sorting is a combinatorial algorithm that sorts points in Euclidean space into layers according to a partial order. It was recently shown that nondominated sorting of random points has a Hamilton-Jacobi equation continuum limit. The original proof, given in \cite{calder2014}, relies on a continuum variational problem. In this paper, we give a new proof using a direct verification argument that completely avoids the variational interpretation. We believe this proof is new in the homogenization literature, and may be generalized to apply to other stochastic homogenization problems for which there is no obvious underlying variational principle.
\end{abstract}

\section{Introduction}
\label{sec:intro}

Many problems in science and engineering require the sorting, or ordering, of large amounts of multivariate data. Since there is no canonical linear criterion for sorting data in dimensions greater than one, many different methods for sorting have been proposed to address various problems (see, e.g., \cite{barnett1976ordering,small1990survey,liu1999multivariate,deb2002}). Many of these algorithms abandon the idea of a linear ordering, and instead sort the data into layers according to some set of criteria.   

We consider here \emph{nondominated sorting}, which arranges a set of points in Euclidean space into layers by repeatedly removing the set of minimal elements. Let $\leqq$ denote the coordinatewise partial order on $\R^d$ defined by
\[x \leqq y \ \ \ \iff \ \ \ x_i \leq y_i \ \ \text{for all } i.\]
The \emph{first nondominated layer}, also called the \emph{first Pareto front} and denoted $\F_1$, is exactly the set of minimal elements of $S$ with respect to $\leqq$, and the deeper fronts are defined recursively as follows:
\[\F_k = \text{Minimal elements of } S\setminus \bigcup_{i<k} \F_i.\]
This peeling process eventually exhausts the entire set $S$, and the result is a partition of $S$ based on Pareto front index, which is often called \emph{Pareto depth} or \emph{rank}. Figure \ref{fig:demo} gives an illustration of nondominated sorting of a random set $S$.
\begin{figure}
\centering
\subfigure[\emph{i.i.d.}~samples]{\includegraphics[clip = true, trim = 30 20 30 20, height=0.24\textwidth]{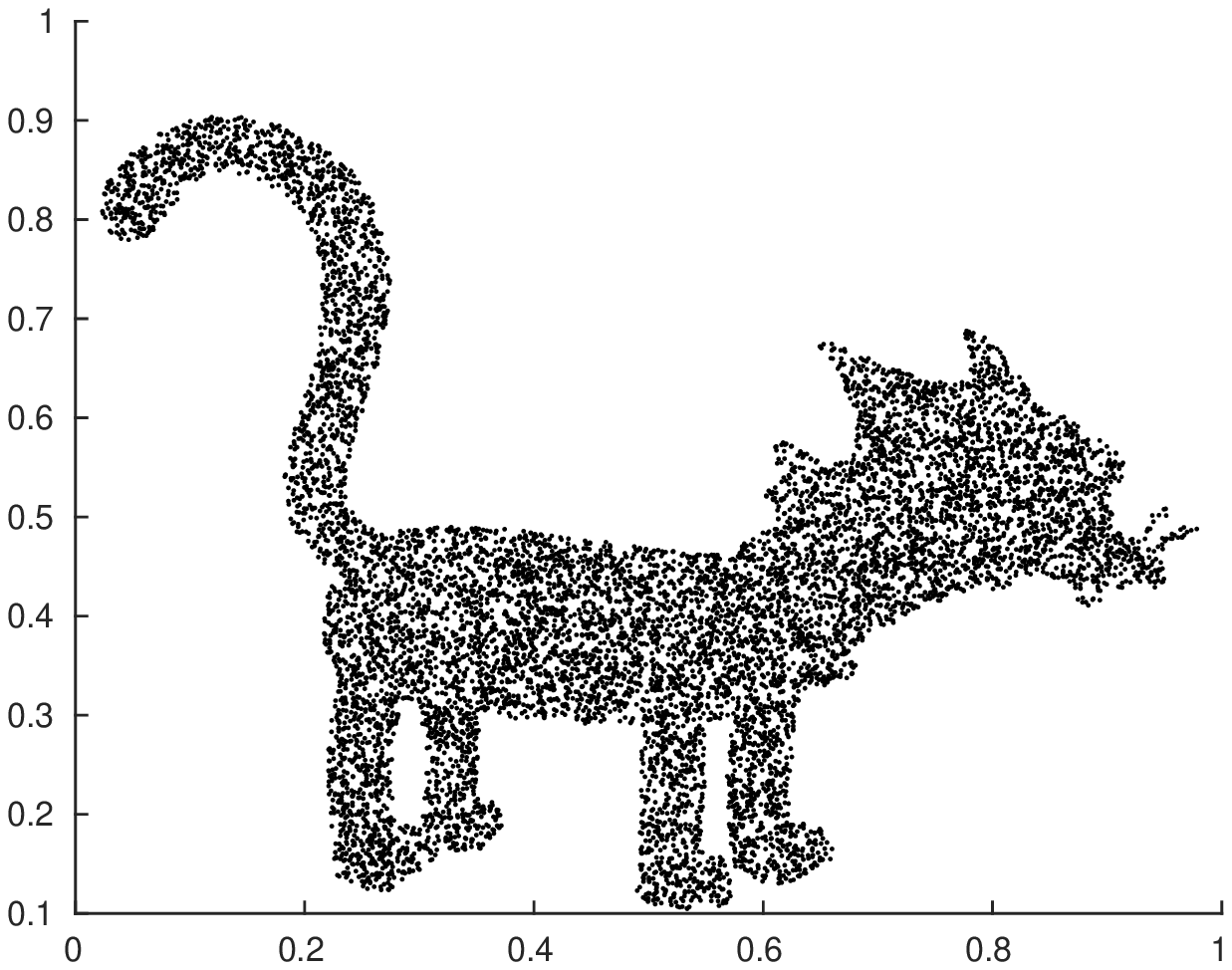}\label{fig:example-fronts}}
\subfigure[$n=10^4$ points]{\includegraphics[clip = true, trim = 30 20 30 20, height=0.24\textwidth]{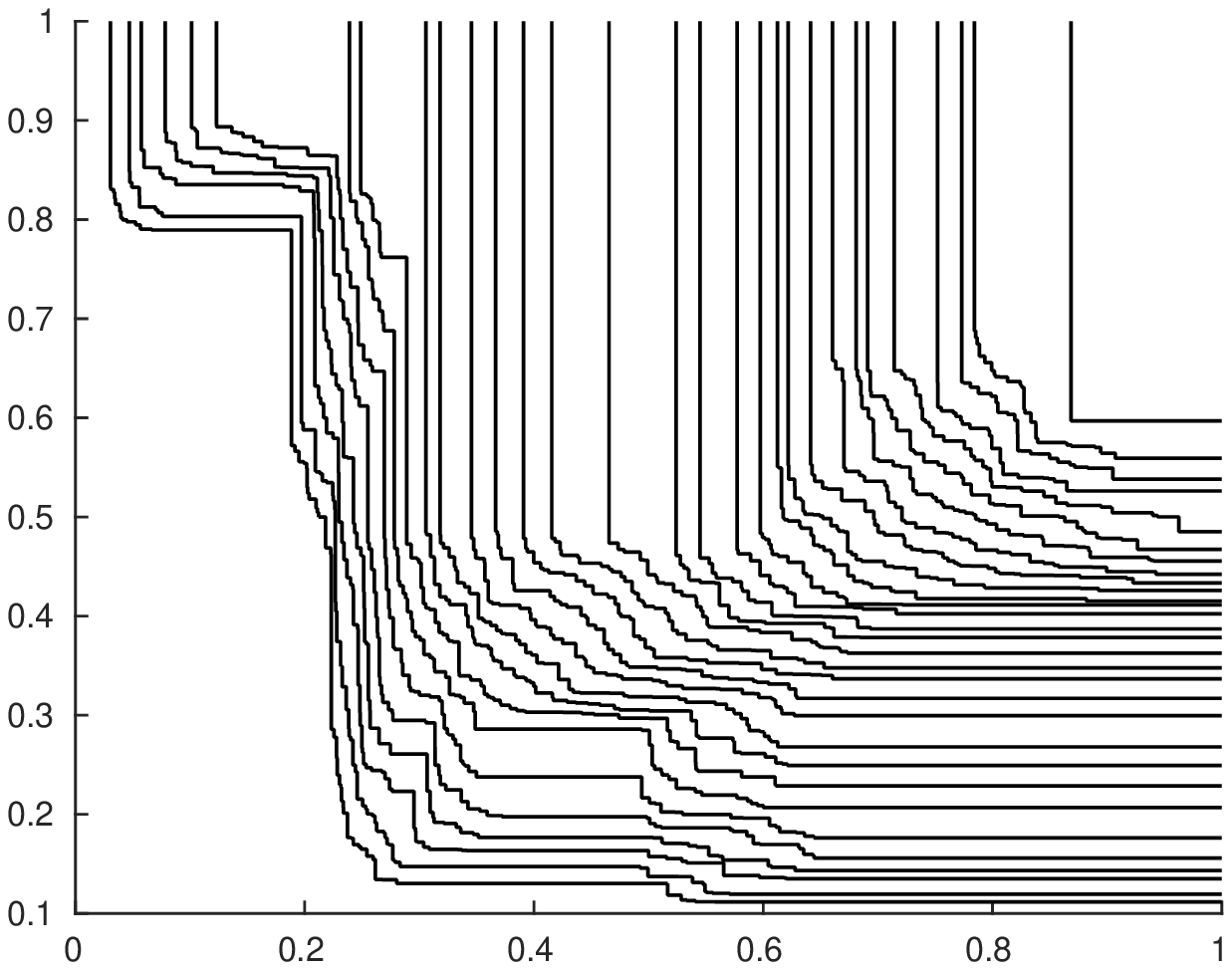}\label{fig:example-fronts-10^4}}
\subfigure[$n=10^6$ points]{\includegraphics[clip = true, trim = 30 20 30 20, height=0.24\textwidth]{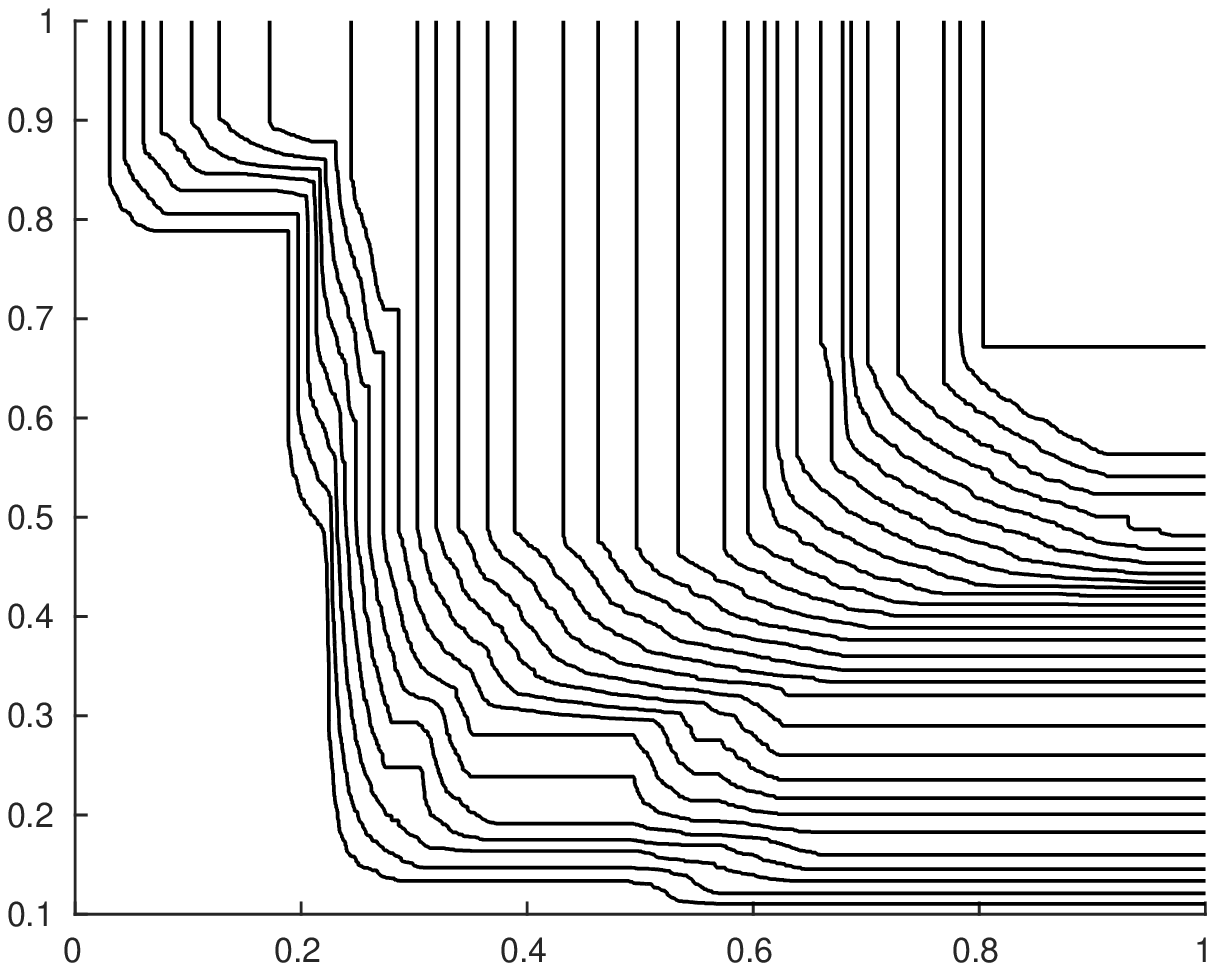}\label{fig:example-fronts-10^6}}
\caption{Examples of Pareto fronts corresponding to \emph{i.i.d.}~random variables $X_1,\dots X_n$ drawn from the distribution depicted in (a). In (b) and (c), we show 30 evenly spaced Pareto fronts for $n=10^4$ and $n=10^6$, respectively.}
\label{fig:demo}
\end{figure}

Nondominated sorting is widely used in multi-objective optimization, where it is the basis of the popular and effective genetic and evolutionary algorithms~\cite{deb2002,fonseca1993,fonseca1995,deb2001,srinivas1994}. Of course, multi-objective optimization is ubiquitous in engineering and scientific contexts, such as control theory and path planning~\cite{mitchell2003,kumar2010,madavan2002}, gene selection and ranking~\cite{speed2003,hero2002,hero2003,hero2004,fleury2002,fleury2004a,fleury2004b}, data clustering~\cite{handl2005}, database systems~\cite{kossmann2002,papadias2005,hsiao2015}, image processing and computer vision~\cite{mumford1989,chan2001}, and some recent machine learning problems~\cite{hsiao2015,hsiao2012,hsiao2015b}.

Nondominated sorting is also equivalent to the longest chain problem, which has a long history in probability and combinatorics~\cite{ulam1961,hammersley1972,bollobas1988,deuschel1995}. A \emph{chain} in $\R^d$ is a finite sequence of points that is totally ordered with respect to $\leqq$. Let $X_1,\dots,X_n$ be $n$ distinct points in $\R^d$ and define 
\begin{equation}\label{eq:unorig}
U_n(x) = \ell(\{X_1,\dots,X_n\} \cap [0,x]),
\end{equation}
where $\ell(\O)$ denotes the length of a longest chain in the finite set $\O \subset \R^d$. 
The notation $[0,x]$ is a special case of the more general interval notation
\[ [x,z] = \big\{y \in \R^d \, : x \leqq y \leqq z\big\} = \prod_{i=1}^d [x_i,z_i]\]
that we shall use throughout the paper. Let $\F_1,\F_2,\F_3,\dots$ denote the Pareto fronts obtained by applying nondominated sorting to $S:=\{X_1,\dots,X_n\}$. Then $x \in \F_1$ if and only if there are no other points $y \in S$ with $y \leqq x$, i.e., $U_n(x) = 1$. A point $x\in S$ is on the second Pareto front $\F_2$ if and only if all points $y \in S$ satisfying $y \leqq x$ are on the first front, and one such point exists. For any such $y \in \F_1$, $\{y,x\}$ is a chain of length $\ell=2$ in $S\cap [0,x]$, and we see that $x \in \F_2 \iff U_n(x) = 2$. Peeling off successive Pareto fronts and repeating this argument yields
\[x \in \F_k \ \ \iff \ \ U_n(x) = k.\]
Hence, the Pareto fronts $\F_1,\F_2,\F_3,\dots$ are embedded into the level sets (or jump sets) of the longest chain function $U_n$, as depicted in Figure \ref{fig:demo}.

In \cite{calder2014}, we proved the following continuum limit for nondominated sorting.
\begin{theorem}\label{thm:prev}
Let $X_1,X_2,X_3,\dots$ be a sequence of i.i.d.~random variables in $\R^d$ with density $f$, and define $U_n$ by \eqref{eq:unorig}. Suppose there exists an open and bounded set $\O \subset \R^d_+$ with Lipschitz boundary such that $f$ is uniformly continuous on $\O$, and $f(x)=0$ for $x\not\in \O$.  Then $n^{-\frac{1}{d}}U_n \longrightarrow u$ uniformly on $[0,\infty)^d$ with probability one, where $u$ is the unique nondecreasing\footnote{We say $u$ is nondecreasing if $x_i \mapsto u(x)$ is nondecreasing for all $i$.} viscosity solution of the Hamilton-Jacobi equation
\[\text{\emph{(P)}} \ \ \left\{ \begin{aligned}
u_{x_1} \cdots u_{x_d} &= (c_d/d)^df& &\text{in } \R^d_+, \\
u&=0& &\text{on } \partial \R^d_+.\end{aligned}\right.\]
and $c_d>0$ are the universal constants given in \cite{bollobas1988}.
\end{theorem}
This result shows that nondominated sorting of massive datasets, which are common in big data applications~\cite{fortin2013generalizing,jensen2003,gilbert2007}, can be approximated by solving a partial differential equation (PDE). In \cite{calder2015PDE}, we exploited this idea to propose a fast approximate nondominated sorting algorithm based on estimating $f$ from the data $X_1,\dots,X_n$ and solving (P) numerically. 

The proof of Theorem \ref{thm:prev}, given in \cite{calder2014}, is based on the following continuum variational problem:
\begin{equation}\label{eq:var}
u(x) = c_d \,\sup \int_0^1 f(\gamma(t))^\frac{1}{d} (\gamma_1'(t) \cdots \gamma_d'(t))^\frac{1}{d} dt,
\end{equation}
where the supremum is over all $C^1$ curves $\gamma:[0,1] \to [0,\infty)^d$ that are monotone nondecreasing in all coordinates and satisfy $\gamma(1)=x$. This can be viewed as a type of stochastic homogenization of the longest chain problem, where the curves $\gamma$ are continuum versions of chains. The variational problem \eqref{eq:var} first appeared in two dimensions in \cite{deuschel1995}. The PDE (P) arises as the Hamilton-Jacobi-Bellman equation associated with \eqref{eq:var} \cite{bardi1997}. We used the same ideas to prove a similar Hamilton-Jacobi equation continuum limit for the directed last passage percolation model in statistical physics, which also has an obvious discrete variational formulation~\cite{calder2015directed}.

It is possible to view Theorem \ref{thm:prev} in the context of stochastic homogenization of Hamilton-Jacobi equations. Indeed, $U_n$ can be interpreted as a discontinuous viscosity solution of the Hamilton-Jacobi equation
\begin{equation}\label{eq:pp}
\left.\begin{aligned}
U_{n,x_1} \cdots U_{n,x_d} &= \sum_{i=1}^n \delta_{X_i}& &\text{in } \R^d_+, \\
U_n&=0& &\text{on } \partial \R^d_+.
\end{aligned}\right\}\end{equation}
Since the right hand side of \eqref{eq:pp} is highly singular, it is worth discussing in what sense we interpret $U_n$ to be a solution of \eqref{eq:pp}. The obvious approach is to mollify the right hand side to obtain a sequence $U_n^\eps$ of approximate solutions. It is possible so show that as $\eps\to 0$, the sequence $U_n^\eps$ converges pointwise to $CU_n$, where the constant $C>0$ depends on the mollification kernel used. Intriguing as this observation is, it is unclear whether \eqref{eq:pp} can be used directly to prove Theorem \ref{thm:prev}.

There is a growing literature on stochastic homogenization of Hamilton-Jacobi equations. Many of the proofs are based on applying subadditive ergodic theory to a representation formula for the solution, like the Hopf-Lax formula, or more generally, a control theoretic interpretation \cite{lions2005homogenization,kosygina2006stochastic,kosygina2008homogenization,souganidis1999,lions2003correctors,rezakhanlou2002continuum,seppalainen1998a,schwab2009stochastic}. These techniques are similar in spirit to our original proof of Theorem \ref{thm:prev} that was based on the variational representation of $u$ given in \eqref{eq:var}. Recently, there has been significant interest in stochastic homogenization techniques that do not require a variational interpretation \cite{lions2010stochastic,armstrong2013stochastic,armstrong2012stochastic,armstrong2015stochastic,rezakhanlou2001continuum}.

In this paper we give an alternative proof of Theorem \ref{thm:prev} that uses a direct verification argument, and completely avoids using the variational interpretation of (P).  The argument is based on a heuristic derivation of (P) that originally appeared in \cite{calder2014}. We review this argument, which is reminiscent of dimensional analysis, in Section \ref{sec:informal}. We believe this proof is new in the stochastic homogenization literature, and it seems to be more robust than arguments based on variational principles. We present the proof in the simplest setting in this paper, but we believe it can be substantially generalized. In future work, we plan to apply this proof technique to other stochastic growth models that do not stem from underlying variational principles. 

\section{Main result}
\label{sec:main}

To simplify the presentation, we model the data here using a Poisson point process. Given a nonnegative function $f \in L^1_{loc}(\R^d)$, we denote by $\ppp{f}$ a Poisson point process with intensity function $f$. This means that $\ppp{f}$ is a random at most countable subset of $\R^d$, and for every bounded Borel set $A \subset \R^d$, the cardinality of $\ppp{f}\cap A$, denoted $N(A)$, is a Poisson random variable with mean $\int_A f \, dx$. Furthermore, for disjoint $A$ and $B$, the random variables $N(A)$ and $N(B)$ are independent. It is worthwhile to mention that in the special case where $\int_{\R^d} f \, dx=1$, one way to construct $\ppp{nf}$ for any $n >0$ is by setting
\begin{equation}\label{eq:cons}
\ppp{nf}:=\{X_1,\dots,X_N\},
\end{equation}
where $X_1,X_2,X_3,\dots$ is an \iid~sequence of random variables with density $f$ and $N$ is a Poisson random variable with mean $n$.  For more details on Poisson point processes, we refer the reader to Kingman's book~\cite{kingman1992poisson}. 

Set
\[\B = \Big\{f:\R^d \to [0,\infty) \, : \,  f \text{ is locally bounded and Lebesgue measurable.}\Big\}.\]
For $f \in \B$, $x \in \R^d$ and $n\in \N$ we define
\begin{equation}\label{eq:defun}
U_n(x) = \ell\left( \ppp{nf}\cap [0,x]\right).
\end{equation}
In this paper, we give a direct verification argument proof of the following result.
\begin{theorem}\label{thm:cont}
Let $f \in \B$ such that (P) has a unique nondecreasing viscosity solution $u$. Then with probability one
\[n^{-\frac{1}{d}}U_n \longrightarrow u \ \ \text{locally uniformly on } [0,\infty)^d.\]
\end{theorem}
\begin{figure}
\centering
\subfigure[$n=50$ points]{\includegraphics[clip = true, trim = 30 20 30 20, height=0.25\textwidth]{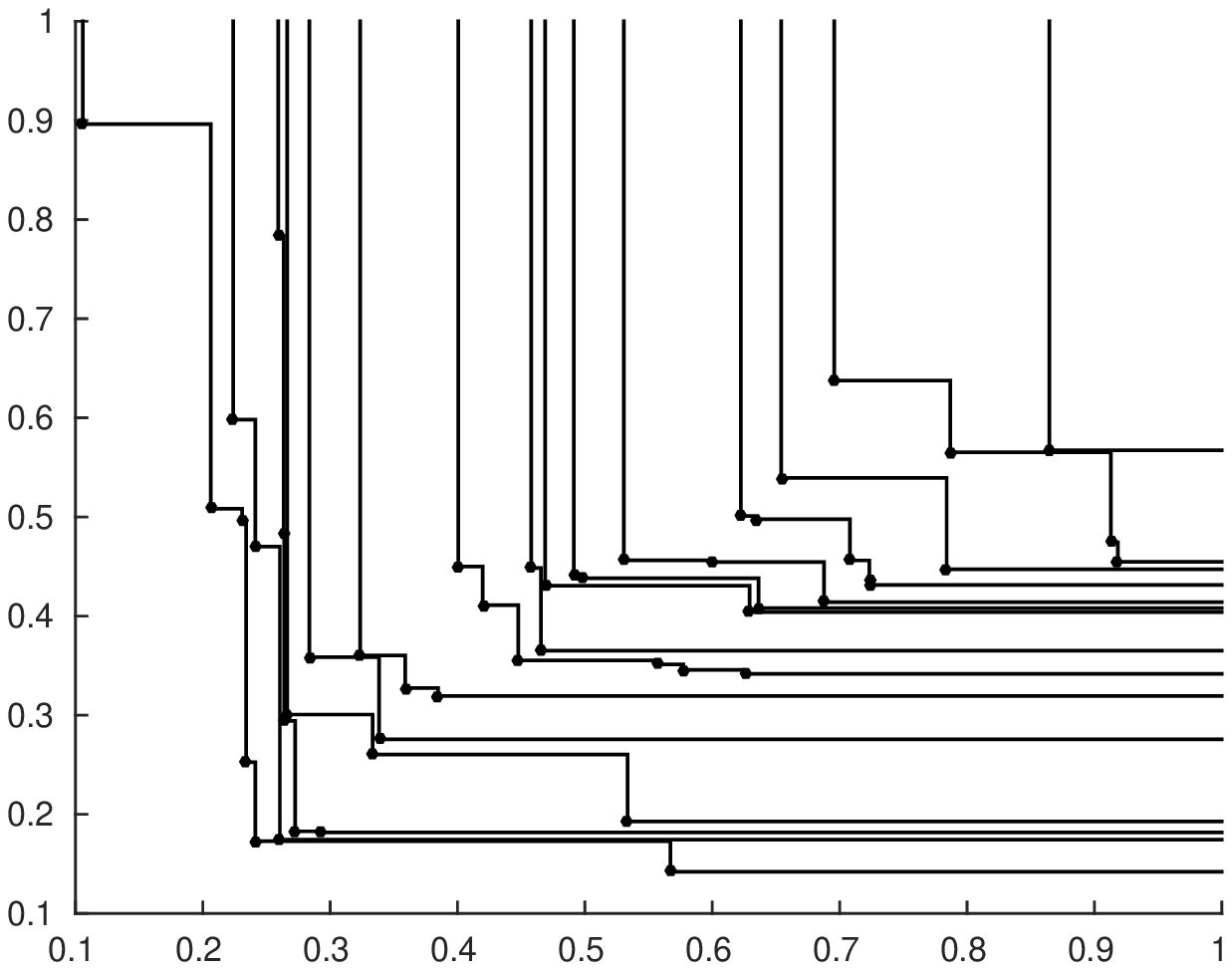}}
\subfigure[$n=10^4$ points]{\includegraphics[clip = true, trim = 30 20 30 20, height=0.25\textwidth]{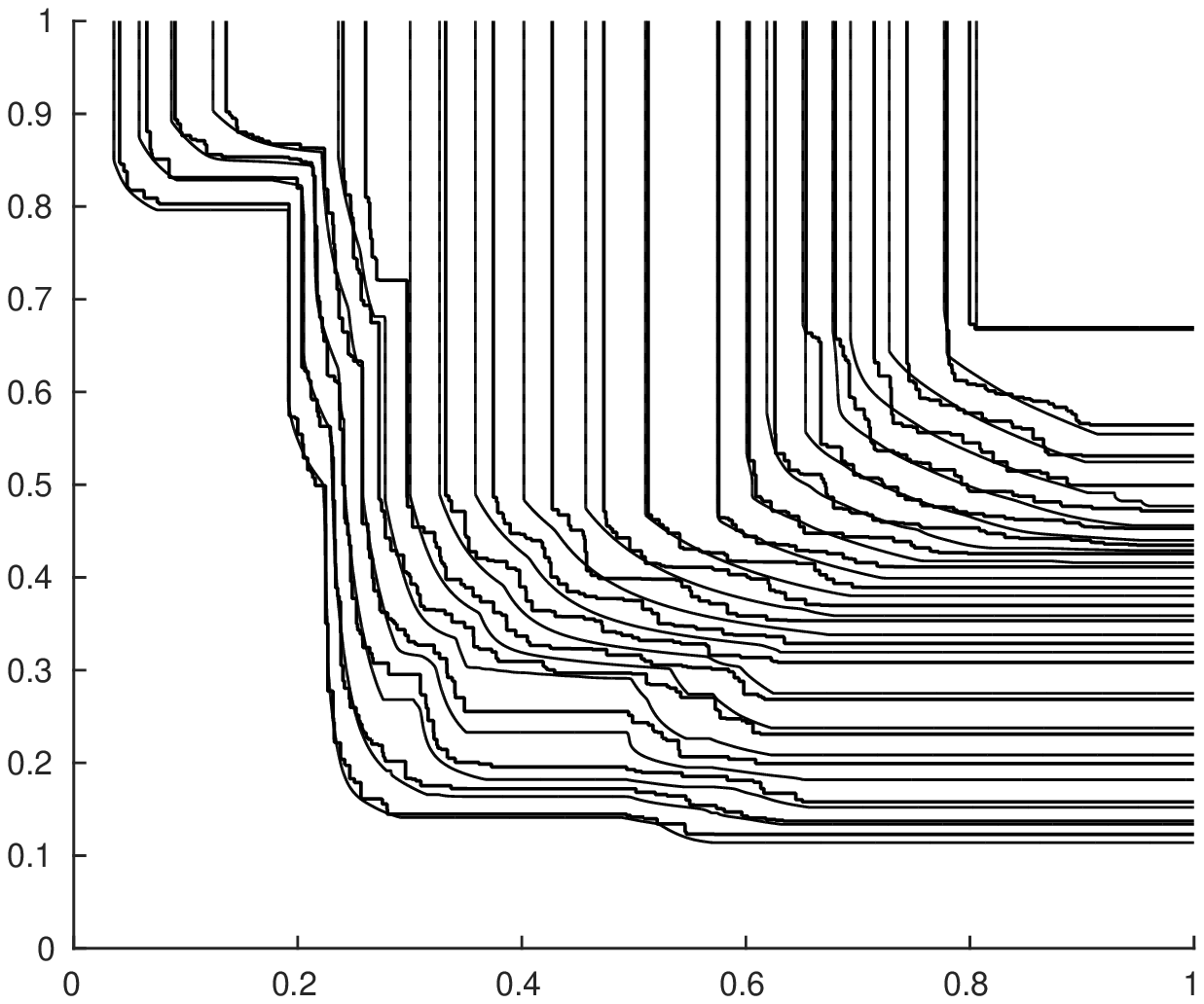}}
\subfigure[$n=10^6$ points]{\includegraphics[clip = true, trim = 30 20 30 20, height=0.25\textwidth]{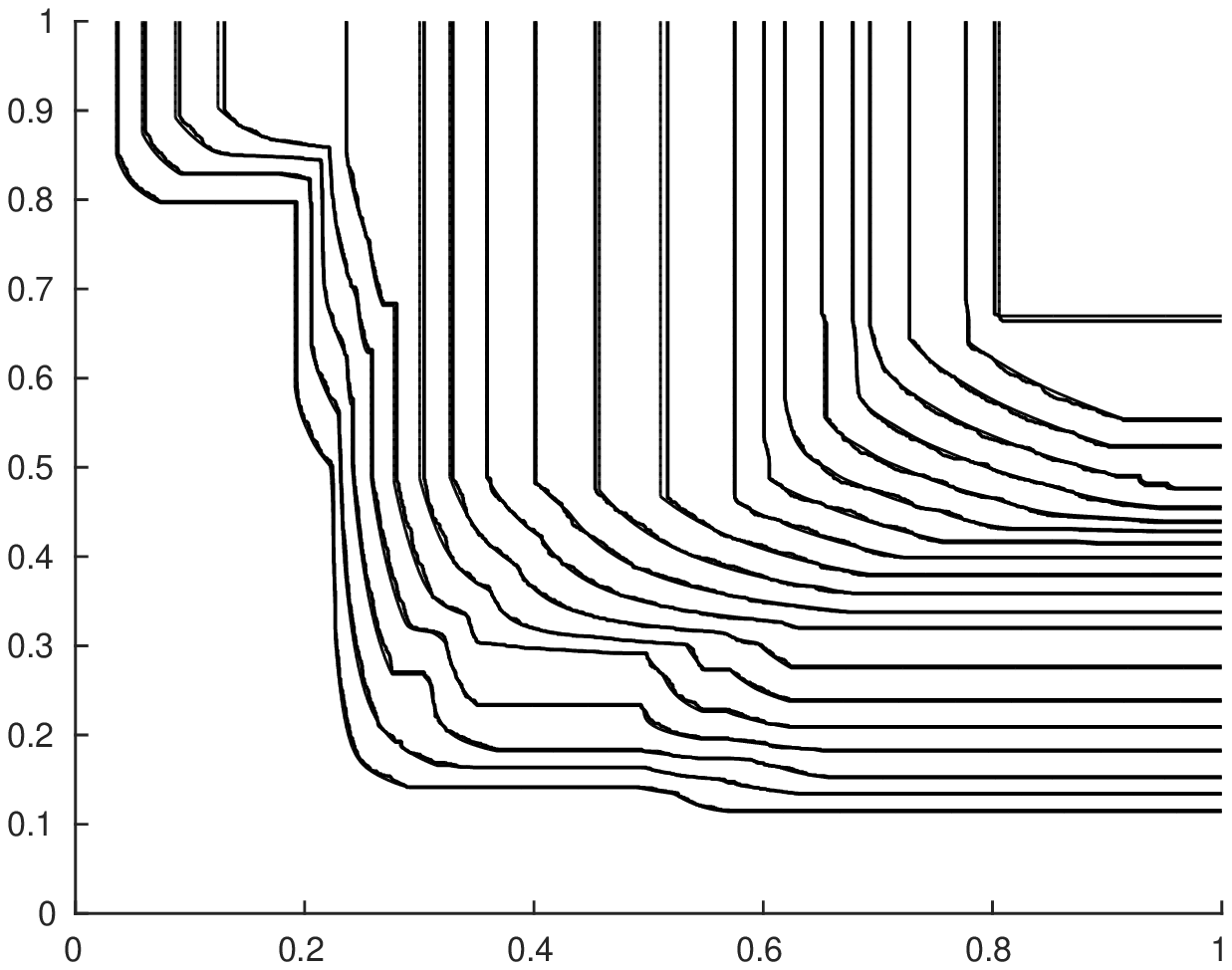}}
\caption{Visualization of the continuum limit given by Theorem \ref{thm:cont} applied to the distribution in Figure \ref{fig:demo}. In (b) and (c), we compare the Pareto fronts to the level sets of the viscosity solution of (P). }
\label{fig:demo2}
\end{figure}

In \cite{calder2014}, we showed that there exists a unique nondecreasing viscosity solution of (P) whenever $f$ is uniformly continuous on an open and bounded set $\O \subset \R^d_+$ with Lipschitz boundary, and $f(x)=0$ for $x \not\in \O$. In \cite{calder2015directed} we extended this uniqueness result to the case where $f$ is continuous on $[0,\infty)^d$, which allows the support of $f$ to be non-compact.  We expect uniqueness to hold under weaker regularity conditions on $f$. Figure \ref{fig:demo2} gives a visual illustration of Theorem \ref{thm:cont}.

In the proof of Theorem \ref{thm:cont}, we assume the family of Poisson point processes $\{\ppp{nf}\}_{n\in \N}$ is defined on a common probability space and we apply deterministic PDE arguments to the individual realizations 
\[U_n^\omega(x) = \ell(\ppp{nf}^\omega\cap [0,x]).\]
In the literature on the longest chain problem, it is common to take a construction for which $n\mapsto U_n^\omega$ is monotone nondecreasing for all realizations $\omega$. Indeed, this was necessary to obtain almost sure convergence in some early works (see \cite{bollobas1988,kesten1973comment}). Due to the sharp concentration of measure results established by Frieze \cite{frieze1991length}, Bollob\'as and Brightwell \cite{bollobas1992height}, and more recently by Talagrand \cite{talagrand1995concentration}, the monotonicity property is no longer required for almost sure convergence. Thus, our proof of Theorem \ref{thm:cont} does not assume any particular construction of $\{\ppp{nf}\}_{n\in \N}$ (see Theorem \ref{thm:complete}).

We should mention, however, that if we do construct a family of Poisson point processes $\{\ppp{tf}\}_{t \in \R_+}$ on a common probability space in such a way that
\begin{equation}\label{eq:mon}
\ppp{sf} \subset \ppp{tf} \ \ \text{ whenever } s \leq t,
\end{equation}
then the function $U_t(x) = \ell(\ppp{tf}\cap[0,x])$ is monotone nondecreasing as a function of $t \in \R_+$ for \emph{every} realization. It then follows directly from Theorem \ref{thm:cont} that 
\[t^{-\frac{1}{d}}U_t \longrightarrow u \ \ \text{locally uniformly on } [0,\infty)^d\]
with probability one as $t\to \infty$. One way to ensure that \eqref{eq:mon} is satisfied is to construct $\ppp{tf}$ via a generalization of \eqref{eq:cons}~\cite{kingman1992poisson}.

\subsection{Informal derivation}
\label{sec:informal}

As motivation, let us give an informal derivation of the Hamilton-Jacobi PDE continuum limit (P). This derivation was originally given in a slightly different form in \cite{calder2014}, and our proof of Theorem \ref{thm:cont} is based on these heuristics. Suppose $f:\R^d\to[0,\infty)$ is continuous and that $n^{-\alpha}U_n \longrightarrow u \in C^1(\R^d)$ uniformly with probability one for some $\alpha \in (0,1]$. The argument below not only derives (P), but also suggests the order of growth $\alpha=\frac{1}{d}$.

Let $x \in (0,\infty)^d$. Since $U_n$ is nondecreasing, the uniform limit $u$ must be nondecreasing and so $u_{x_i}(x) \geq 0$ for all $i$. Let us assume that $u_{x_i}(x)>0$ for all $i$.  Fix $v \in \R^d$ with $\langle Du(x), v\rangle > 0$ and consider the quantity $n^\alpha (u(x+v) - u(x))$.  This is approximately the number of Pareto fronts passing between $x$ and $x+v$ when $n$ is large.  When counting these fronts, we may restrict ourselves to Poisson points falling in the set
\begin{equation}\label{eq:usimplex}
A =\big\{y \in \R^d \, : \, y \leqq x + v \ \text{ and } \  u(y) \geq u(x) \big\}.
\end{equation}
This is because any points in $\{y \in \R^d \, : \, u(y) < u(x)\}$ will be on  previous Pareto fronts and only points that are coordinatewise less than $x+v$ can influence the Pareto rank of $x+v$. See Figure \ref{fig:derivation} for a depiction of this region and some quantities from the derivation.
\begin{figure}
\centering
\fbox{\includegraphics[width=0.4\textwidth]{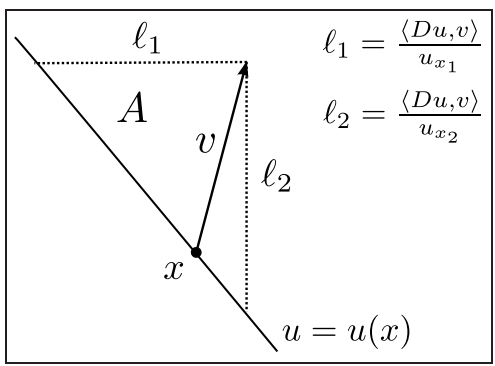}}
\caption{Some quantities from the informal derivation of the Hamilton-Jacobi PDE (P).}
\label{fig:derivation}
\end{figure}
Since $u_{x_i}(x) > 0$ for all $i$, and $u$ is $C^1$, $A$ is well approximated for small $|v|$ by a simplex with orthogonal corner, and the points within $A$ are approximately uniformly distributed.  Let $N=N(A)$ denote the number of Poisson points falling in $A$. Since we can scale the simplex into a standard simplex while preserving the Pareto fronts within $A$, it is reasonable to conjecture that the number of Pareto fronts within $A$ (or the length of a longest chain in $A$) is asymptotic to $c_1N^\alpha$ for some constant $c_1>0$, independent of $x$.  

By the law of large numbers, we have $N \approx n\int_A f(y) \, dy$.  Hence when $|v|>0$ is small 
\begin{equation}\label{eq:usim}
n^\alpha(u(x+v)-u(x)) \approx c\left(n\int_A f(y) \, dy\right)^{\alpha} \approx c_1n^{\alpha}|A|^{\alpha}f(x)^\alpha,
\end{equation}
where $|A|$ denotes the Lebesgue measure of $A$.
Let $\ell_1,\dots,\ell_d$ denote the side lengths of the simplex $A$.  Then $|A|\approx c_2 \, \ell_1\cdots \ell_d$ for a constant $c_2$.  Since $x+v-\ell_i e_i$ lies approximately on the tangent plane to the level set $\{y \, : \, u(y)=u(x)\}$, we see that
\[\langle Du(x), v - \ell_i e_i\rangle \approx 0.\]
This gives $\ell_i \approx u_{x_i}(x)^{-1}\langle D u(x),v\rangle$, and hence
\begin{equation}\label{eq:measureA}
|A|\approx c_2 (u_{x_1}(x)^{-1} \cdots u_{x_d}(x)^{-1}) \langle Du(x),v\rangle^d.
\end{equation}
For small $|v|$, we can combine \eqref{eq:usim} and \eqref{eq:measureA} to obtain
\[\langle Du(x), v\rangle \approx u(x+v) - u(x) \approx c_1 c_2^\alpha f(x)^\alpha (u_{x_1}(x)^{-\alpha} \cdots u_{x_d}(x)^{-\alpha}) \langle Du(x),v\rangle^{\alpha d}.\]
Therefore
\[u_{x_1}(x)^\alpha \cdots u_{x_d}(x)^\alpha \approx Cf(x)^\alpha\langle Du(x),v\rangle^{\alpha d - 1},\]
where $C = c_1c_2^\alpha$. Since the left hand side is independent of $v$, this suggests that $\alpha=\frac{1}{d}$ and 
\[u_{x_1} \cdots u_{x_d} = C^df \ \ \text{ on } \R^d.\]
  
Although this derivation is informal, it is straightforward and conveys the essence of the result. Inspecting the argument, we see that the main ingredients are an invariance with respect to scaling along coordinate axes, and some basic properties of the coordinatewise partial order. It is intriguing that we can derive the order of growth $n^\frac{1}{d}$ and the continuum limit (P) from such basic information. This seems analogous to dimensional analysis arguments in applied mathematics, which use ideas like self-similarity and scaling to derive fundamental laws for natural phenomena~\cite{barenblatt2003scaling}. A fundamental feature of dimensional analysis is that, like our argument above, it usually gives no information about the constants of proportionality. Contrary to dimensional analysis, which is typically nonrigorous, we show in this paper that the argument above can be made rigorous using the framework of viscosity solutions.

\subsection{Outline of proof}
\label{sec:outline}

Our new proof of Theorem \ref{thm:cont} is a rigorous version of the heuristic dimensional analysis argument given in Section \ref{sec:informal}. The main ideas used in our proof draw some inspiration from the Barles-Souganidis framework for convergence of numerical schemes \cite{barles1991}. The Barles-Souganidis framework establishes convergence of any numerical scheme that is monotone, consistent, and stable, provided the PDE admits a strong comparison principle. In our proof of Theorem \ref{thm:cont}, we view $U_n$ as a numerical approximation of the viscostiy solution of (P), and use ideas that parallel monotonicity, stability and consistency to prove convergence to the viscosity solution of (P). 

The longest chain problem is monotone in the following sense: For any finite sets $A,B\subset \R^d$,
\begin{equation}\label{eq:monotonicity}
A\subset B \implies \ell(A) \leq \ell(B).
\end{equation}
When proving convergence of numerical schemes to viscosity solutions, monotonicity of the scheme is used to replace the numerical solution by a smooth test function for which consistency is trivial. The monotonicity property \eqref{eq:monotonicity} is used for precisely the same purpose in the proof of Theorem \ref{thm:cont}, though consistency is non-trivial in this setting.

Stability of numerical schemes for PDE usually refers to a uniform bound on the numerical solutions that is independent of the grid resolution. Here, we use a type of asymptotic equicontinuity. Namely, in Theorem \ref{thm:holder2} we show that a bound of the form
\begin{equation}\label{eq:stability}
|U_n(x) - U_n(y)| \lesssim C|x - y|^\frac{1}{d}n^\frac{1}{d}
\end{equation}
holds for all $x,y$ with probability one. The stability estimate \eqref{eq:stability} and the monotonicity of $U_n$ can be combined with the Arzel\`a-Ascoli theorem to show that for every realization $\omega$ in a probability one event $\Omega$, there exists a subsequence $n_k^{-\frac{1}{d}}U^\omega_{n_k}$ converging locally uniformly on $[0,\infty)^d$ to a locally  H\"older-continuous function $U^\omega$.

The final step in the proof is to use consistency and monotonicity to identify $U^\omega$ as the unique nondecreasing viscosity solution of (P).  Our consistency result, Theorem \ref{thm:con}, is a statement on the asymptotic length of a longest chain in sets resembling the approximate simplex \eqref{eq:usimplex} used in the heuristic derivation. Namely, for a smooth test function $\phi$ with $\phi_{x_i}>0$ for all $i$, we prove asymptotics of the form
\[\ell( \ppp{nf}\cap A_\eps) \sim \left(\frac{c_d}{d} \left(\frac{f(x)}{\phi_{x_1}(x)\cdots \phi_{x_d}(x)}\right)^\frac{1}{d}\eps + O(\eps^2)\right)n^\frac{1}{d},\]
as $n\to \infty$, where
\[A_\eps(x) = \Big\{ y \in B(x,\sqrt{\eps}) \, : \, y \leqq x \text{ and } \phi(y) \geq \phi(x) - \eps\Big\}.\]
Notice the set $A_\eps(x)$ is similar to the set $A$ used in the heuristic derivation and defined in \eqref{eq:usimplex}. The main difference is that $A_\eps(x)$ is defined using a smooth test function, and more importantly, it is defined by looking ``backward'' from the point $x$, instead of looking ``forward'' as in \eqref{eq:usimplex}. This respects the natural domain of dependence of the nondominated sorting problem, and allows us to use the monotonicity of $\ell$ in the proof of Theorem \ref{thm:cont}. This is analogous to the idea of an upwind numerical scheme for Hamilton-Jacobi equations.

\section{Analysis of longest chain problem}
\label{sec:analysis}

Let us give a brief history of the longest chain problem. Let $X_1,\dots,X_n$ be \iid~random variables uniformly distributed on the unit hypercube $[0,1]^d$, and let $\ell_n = \ell(\{X_1,\dots,X_n\})$ be the length of a longest chain.  Hammersley \cite{hammersley1972} was the first to study the asymptotics of $\ell_n$, motivated by a connection to Ulam's famous problem of finding the longest monotone subsequence in a random sequence of numbers \cite{ulam1961}. Hammersley showed that in dimension $d=2$, $n^{-\frac{1}{2}} \ell_n \to c>0$ in probability as $n\to \infty$, and he conjectured that $c=2$. The value $c=2$ was later established in \cite{logan1977,vershik1977}. Hammersley's proof is now a classic application of subadditive ergodic theory. 

Bollob\'as and Winkler \cite{bollobas1988} extended Hammersley's result to dimensions $d\geq 3$ showing that  $n^{-\frac{1}{d}}\ell_n \to c_d>0$ in probability. They also established the bounds
\[\frac{d^2}{d!^\frac{1}{d}\Gamma\left(\frac{1}{d}\right)} \leq c_d < e,\]
which by Stirling's formula shows that $\lim_{d\to \infty} c_d = e$. Aside from $c_2=2$, the exact values of $c_d$ are still unknown. When the random variables $\{\ell_n\}_{n\in \N}$ are defined on a common probability space in such a way that $n \mapsto \ell_n$ is monotone nondecreasing along sample paths, the stronger almost sure convergence can be obtained via an observation of Kesten \cite{kesten1973comment}.

Here, we consider the longest chain among Poisson points. For $t>0$ let
\[L_t = \ell(\ppp{t}\cap[0,1]^d).\]
Then $L_t$ has the same distribution as 
\[\bar{L}_t = \ell(\ppp{1} \cap [0,t^\frac{1}{d}]^d).\]
Notice that the family $\{\bar{L}_t\}_{t>0}$ is defined on a common probability space and $t\mapsto \bar{L}_t$ is monotone nondecreasing along sample paths. Subadditive ergodic theory \cite{hammersley1972,kingman1973subadditive} can be employed to show that 
\[t^{-\frac{1}{d}}\bar{L}_t \to c_d \ \ \text{almost surely as }t\to \infty.\]
Since $L_t$ and $\bar{L}_t$ merely have the same distribution, this argument only gives convergence in probability for $L_t$.

Since we have made no assumptions about how the Poisson point process $\ppp{nf}$ is to be constructed, the proof of our main result (Theorem \ref{thm:cont}) relies on establishing almost sure convergence for sequences of random variables that have the same distribution as $L_t$, but may not necessarily satisfy the monotonicity condition.  Using the concentration of measure results provided by Talagrand's isoperimetric theory \cite{talagrand1995concentration} we can establish the following result.
\begin{theorem}\label{thm:complete}
For $t > 0$, let $L_t = \ell(\ppp{t}\cap[0,1]^d)$. Let $\{t_n\}_{n\in \N}$ be an increasing sequence of positive real numbers such that 
\begin{equation}\label{eq:seqreq}
\sum_{n=1}^\infty \exp\left(-\eps t_n^\frac{1}{d}\right) < \infty \ \ \text{for all } \eps>0.
\end{equation}
Then 
\[t_n^{-\frac{1}{d}}L_{t_n} \longrightarrow  c_d \ \ \text{completely as } n \to \infty.\]
\end{theorem}
 Recall that a sequence of random variables $X_1,X_2,\dots,X_n,\dots$ \emph{converges completely} to a random variable $X$ if 
\[\sum_{n=1}^\infty P(|X_n - X|>\eps) < \infty \ \ \text{for all } \eps>0.\]
By the Borel-Cantelli lemma, complete convergence implies almost sure convergence. The advantage of complete convergence is that it depends only on the distributions of $X_n$ and $X$, and not on any particular construction of the random variables. We believe Theorem \ref{thm:complete} is likely well-known in the probability literature, but to the best of our knowledge it has not been published. For the sake of completeness, we have sketched the proof of Theorem \ref{thm:complete} in the appendix.

\subsection{Stability}
\label{sec:stability}

In this section, we prove in Theorem \ref{thm:holder2} a stability result for the longest chain problem. We first have a preliminary lemma.
\begin{lemma}\label{lem:holder}
Let $f\in \B$, and define $U_n$ by \eqref{eq:defun}. Then for any $x,y \in [0,\infty)^d$
\begin{equation}\label{eq:holder2}
\limsupn n^{-\frac{1}{d}}(U_n(x) - U_n(y)) \leq c_ddM^\frac{d-1}{d}\|f\|_{L^\infty([0,M]^d)}^\frac{1}{d}|x - y|^\frac{1}{d}
\end{equation}
with probability one, where $M=\max\{x_1,y_1,x_2,y_2,\dots,x_d,y_d\}$.
\end{lemma}
\begin{proof}
We first show that
\begin{equation}\label{eq:holder}
\limsupn n^{-\frac{1}{d}}(U_n(x + he_i) - U_n(x)) \leq c_dM^\frac{d-1}{d}\|f\|_{L^\infty([x_ie_i,x+he_i])}^\frac{1}{d}h^\frac{1}{d},
\end{equation}
for all $h>0$ and $i \in \{1,\dots,d\}$.
Without loss of generality we may assume that $i=1$. If $x_j = 0$ for some $j\geq 2$, then $x,x+he_1 \in \partial \R^d_+$, and hence $U_n(x) = U_n(x+he_1) = 0$ with probability one. Thus, we may assume that $x_j >0$ for all $j\geq 2$. 

Let $L = U_n(x+he_1)$ and let $C=\{X_1,\dots,X_L\}$ be a chain in $\ppp{nf}\cap [0,x+he_1]$ of length $L$. We can split this chain into two chains, $C_1$ and $C_2$, such that $C_1$ lies in the rectangle $[0,x]$, and $C_2$ belongs exclusively to $R:=[x_1e_1,x+he_1]$. See Figure \ref{fig:stability} for a depiction of the setup. By the definition of $U_n$
\begin{figure}
\centering
\includegraphics[width=0.6\textwidth]{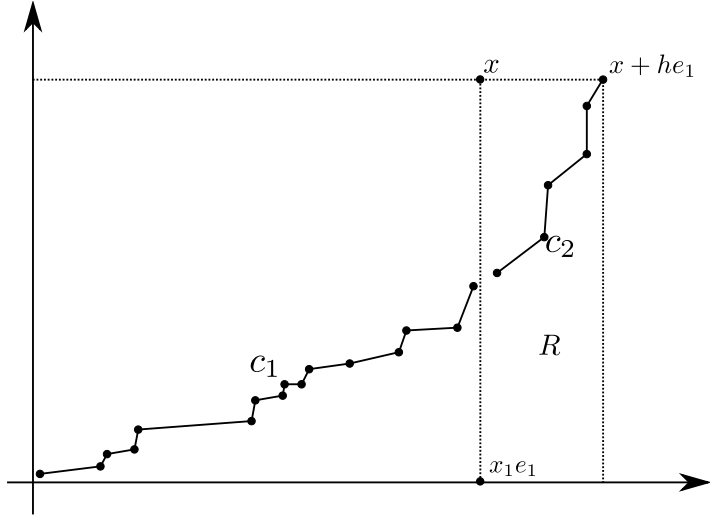}
\caption{A depiction of the chains $C_1$ and $C_2$ from the proof of Lemma \ref{lem:holder}.}
\label{fig:stability}
\end{figure}
\begin{equation}\label{eq:holder-prelim}
U_n(x+he_1) = \ell(C_1) + \ell(C_2) \leq U_n(x) + \ell(\ppp{nf}\cap R).
\end{equation}

Let $K=\sup_R f$ and set 
\[g(x)=\begin{cases}
K-f(x),& \text{if } x \in R \\
0,& \text{otherwise.}\end{cases}\] 
Let $\ppp{ng}$ be a Poisson point process with intensity $ng$ and let $\ppp{}=\ppp{nf}\cup \ppp{ng}$. Then $\ppp{}$ is a Poisson point process with intensity $\lambda$, where
\[\lambda(x)=nf(x)+ng(x)=\begin{cases}
nK,& \text{if } x \in R \\
nf(x),& \text{otherwise.}\end{cases}\] 
Let $N=hx_2x_3\cdots x_d Kn$. By scaling, $\ell(\ppp{}\cap R)$ has the same distribution as $\ell(\ppp{N}\cap[0,1]^d)$. By Theorem \ref{thm:complete} and the inclusion $\ppp{nf}\subset \ppp{}$ we have
\[\limsup_{n\to \infty} n^{-\frac{1}{d}}\ell(\ppp{nf}\cap R)\leq \lim_{n\to \infty} n^{-\frac{1}{d}} \ell(\ppp{}\cap R) = c_d(x_2\cdots x_d)^\frac{1}{d}K^\frac{1}{d} h^\frac{1}{d} \leq c_d M^\frac{d-1}{d} K^\frac{1}{d} h^\frac{1}{d}\]
with probability one.  Combining this with \eqref{eq:holder-prelim} establishes \eqref{eq:holder}.

We can apply \eqref{eq:holder} in each coordinate and note that $\sum |x_i-y_i|^\frac{1}{d} \leq d^{1-\frac{1}{2d}}|x-y|^\frac{1}{d}$ to establish the desired result \eqref{eq:holder2}.
\end{proof}
Recall that $f^*$ and $f_*$ denote the upper and lower semicontinuous envelopes of $f$, respectively, defined by
\[f^*(x) = \limsup_{y \to x} f(y),\]
and $f_* = -(-f)^*$.
\begin{theorem}[Stability]\label{thm:holder2}
Let $f\in \B$, and define $U_n$ by \eqref{eq:defun}. Then with probability one 
\begin{equation}\label{eq:holder-final}
\limsupn n^{-\frac{1}{d}}(U_n(x) - U_n(y)) \leq c_ddM^\frac{d-1}{d}\|f^*\|_{L^\infty([0,M]^d)}^\frac{1}{d}|x-y|^\frac{1}{d} 
\end{equation}
for all $x,y \in [0,\infty)^d$, where $M=\max\{x_1,y_1,x_2,y_2,\dots,x_d,y_d\}$.
\end{theorem}
It is important to emphasize in Theorem \ref{thm:holder2} that \eqref{eq:holder-final} holds with probability one \emph{simultaneously} for all $x,y \in [0,\infty)^d$. This is a far stronger statement than Lemma \ref{lem:holder}, where it was shown that for fixed $x,y \in [0,\infty)^d$, \eqref{eq:holder-final} holds with probability one. This stronger result gives us a form of compactness that is used in the proof of Theorem \ref{thm:cont}.
\begin{proof}
For $x,y \in [0,\infty)^d$ let $\Omega_{x,y}$ denote the event that \eqref{eq:holder2} holds for $x,y$.
 By Lemma \ref{lem:holder}, $\P(\Omega_{x,y})=1$. Let 
\[\Omega = \bigcap \Big\{\Omega_{x,y} \, : \, x,y \in \Q^d\cap[0,\infty)^d\Big\} .\]
Being the countable intersection of probability one events, $\Omega$ has probability one. 

Let $x,y\in[0,\infty)^d$. Let $\hat{x},\hat{y} \in \Q^d\cap[0,\infty)^d$ such that $x \leqq \hat{x}$ and $\hat{y} \leqq y$, and set 
\[\hat{M}=\max\{\hat{x}_1,\hat{y}_1,\hat{x}_2,\hat{y}_2,\dots,\hat{x}_d,\hat{y}_d\}.\]
Then since $U_n$ is nondecreasing \[U_n(x) - U_n(y) \leq U_n(\hat{x}) - U_n(\hat{y}),\]
and therefore
\[\limsupn n^{-\frac{1}{d}}(U_{n}^\omega(x)- U_{n}^\omega(y)) \leq c_dd\hat{M}^\frac{d-1}{d}\|f\|_{L^\infty([0,\hat{M}]^d)}^\frac{1}{d}|\hat{x}-\hat{y}|^\frac{1}{d} \ \ \text{for all } \omega \in \Omega.\]
Since we can choose $\hat{x}$ and $\hat{y}$ arbitrarily close to  $x$ and $y$, respectively, we obtain 
\[\limsupn n^{-\frac{1}{d}}(U_{n}^{\omega}(x)- U_{n}^{\omega}(y)) \leq c_ddM^\frac{d-1}{d}\|f^*\|_{L^\infty([0,M]^d)}^\frac{1}{d}|x-y|^\frac{1}{d} \ \ \text{for all } \omega \in \Omega,\]
which completes the proof.
\end{proof}

\subsection{Consistency}

We now establish our main consistency result. For $x,v \in (0,\infty)^d$, we define 
\begin{equation}\label{eq:gen-simplex}
S_{x,v} = \Big\{ z \in \R^d \, : \, z \leqq x \text{ and } \langle x-z,v\rangle \leq 1\Big\}.
\end{equation}
The set $S_{x,v}$ is a simplex with orthogonal corner at $x$ and side lengths $1/v_1,\dots,1/v_d$. See Figure \ref{fig:simplex} for an illustration of $S_{x,v}$. We also define $\vb{1} = (1,\dots,1) \in \R^d$.
\begin{figure}[t!]
\centering
\includegraphics[width=0.5\textwidth]{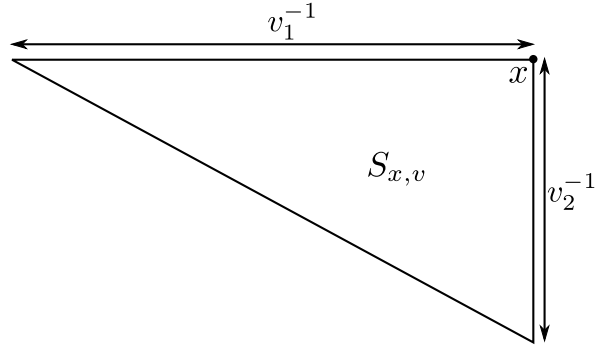}
\caption{Illustration of the simplex $S_{x,v}$.}
\label{fig:simplex}
\end{figure}
\begin{lemma}\label{lem:gen-simplex}
Let $x,v \in (0,\infty)^d$ and let $\{t_n\}_{n\in \N}$ be any increasing sequence of positive real numbers satisfying \eqref{eq:seqreq}. Then 
\begin{equation}\label{eq:gen-simplex-lim}
\lim_{n\to \infty}t_n^{-\frac{1}{d}} \ell\left(\ppp{t_n}\cap S_{x,v} \right) = \frac{c_d}{d(v_1\cdots v_d)^\frac{1}{d}}
\end{equation}
in the sense of complete convergence.
\end{lemma}
\begin{proof}
We first show that
\begin{equation}\label{eq:simplex-lim}
\lim_{n\to \infty}t_n^{-\frac{1}{d}} \ell\left(\ppp{t_n}\cap S_{\vb{1},\vb{1}} \right) = \frac{c_d}{d} \ \ \text{ completely.}
\end{equation}
Note that we can write
\[S_{\vb{1},\vb{1}} = \Big\{ x \in [0,1]^d \, : \, x_1 + \cdots + x_d \geq d-1\Big\}.\]
Let $K$ be a positive integer and define
\[ I = \left\{ x \in [0,1]^d \, : \, Kx \in \Z^d \text{ and } x_1 + \cdots +x_d = d-1-\frac{d}{K}\right\}.\]
It follows from this construction that \emph{every} chain in $S_{\vb{1},\vb{1}}$ is contained entirely in $[x,\vb{1}]$ for some $x \in I$. Since $I$ is a finite set
\begin{equation}\label{eq:maxbound}
\ell\left(\ppp{t}\cap S_{\vb{1},\vb{1}})\right) \leq \max\Big\{ \ell\left(\ppp{t}\cap [x,\vb{1}]\right) \, : \, x \in I\Big\}.
\end{equation}
Fix $x \in I$ and let $N_n =t_n\prod_{i=1}^d(1-x_i)$. Since $\ell(\ppp{t_n}\cap [x,\vb{1}])$ has the same distribution as $\ell(\ppp{N_n}\cap [0,1]^d)$ we have by Theorem \ref{thm:complete} that
\[\lim_{n\to \infty} t_n^{-\frac{1}{d}}\ell(\ppp{t_n}\cap [x,\vb{1}]) = c_d\prod_{i=1}^d(1-x_i)^\frac{1}{d}\leq \frac{c_d}{d}\sum_{i=1}^d(1-x_i) = c_d\left(\frac{1}{d} + \frac{1}{K}\right),\]
in the sense of complete convergence. Combining this with \eqref{eq:maxbound} and noting that $K$ was arbitrary yields
\begin{equation}\label{eq:oneway}
\limsupn t_n^{-\frac{1}{d}}\ell\left(\ppp{t_n}\cap S_{\vb{1},\vb{1}})\right) \leq \frac{c_d}{d}
\end{equation}
completely.

For the other direction, let $z = \vb{1}(d-1)/d$. By Theorem \ref{thm:complete}
\[\liminfn t_n^{-\frac{1}{d}}\ell\left(\ppp{t_n}\cap S_{\vb{1},\vb{1}}\right) \geq \limn t_n^{-\frac{1}{d}} \ell(\ppp{t_n}\cap [z,\vb{1}]) = c_d \prod_{i=1}^d (1-z_i)^\frac{1}{d} = \frac{c_d}{d}.\]
in the sense of complete convergence. This establishes \eqref{eq:simplex-lim}.

We now establish \eqref{eq:gen-simplex-lim} with a simple scaling argument. Let $x,v \in (0,\infty)^d$ and define $\Phi:\R^d \to \R^d$ by
\[\Phi(y) = \left( (y_1-x_1)v_1 + 1,(y_2-x_2)v_2 + 1, \dots,(y_d-x_d)v_d +1\right).\]
Since $v_i>0$ for all $i$, the mapping $\Phi$ preserves the partial order $\leqq$ and therefore
\begin{equation}\label{eq:simplex-chain1}
\ell(\ppp{t_n}\cap S_{x,v}) = \ell\left( \Phi(\ppp{t_n}\cap S_{x,v})\right).
\end{equation}
Let $\ppp{}$ be the Poisson point process induced by the mapping $\Phi$, i.e., $\ppp{} = \Phi(\ppp{n})$. Then $\ppp{}$ is a Poisson point process on $\R^d$ with constant intensity $N_n=t_n/(v_1\cdots v_d)$. Since $\Phi(S_{x,v}) = S_{\vb{1},\vb{1}}$, we have by \eqref{eq:simplex-chain1} that $\ell(\ppp{t_n}\cap S_{x,v})$ and $\ell\left(\ppp{N_n}\cap S_{\vb{1},\vb{1}}\right)$ have the same distribution. The result now follows directly from  \eqref{eq:simplex-lim}.
\end{proof}
\begin{lemma}\label{lem:simplex-f}
Let $f\in \B$. Then for any $x,v \in (0,\infty)^d$
\begin{equation}\label{eq:gen-simplex-fmax}
\limsupn n^{-\frac{1}{d}} \ell\left(\ppp{nf}\cap S_{x,v} \right) \leq\frac{c_d}{d} \left(\frac{\sup_{S_{x,v}} f}{v_1\cdots v_d}\right)^\frac{1}{d},
\end{equation}
and
\begin{equation}\label{eq:gen-simplex-fmin}
\liminfn n^{-\frac{1}{d}} \ell\left(\ppp{nf}\cap S_{x,v} \right) \geq\frac{c_d}{d} \left(\frac{\inf_{S_{x,v}} f}{v_1\cdots v_d}\right)^\frac{1}{d}
\end{equation}
with probability one.
\end{lemma}
\begin{proof}
Let $x,v \in (0,\infty)^d$, $K=\sup_{S_{x,v}} f$, and define
\[g(x)=\begin{cases}
K-f(x),& \text{if } x \in S_{x,v}, \\
0,& \text{otherwise.}\end{cases}\] 
Let $\ppp{ng}$ be a Poisson point process with intensity $ng$ and let $\ppp{}=\ppp{nf}\cup \ppp{ng}$. Then $\ppp{}$ is a Poisson point process with intensity $\lambda$, where
\[\lambda(x)=nf(x)+ng(x)=\begin{cases}
nK,& \text{if } x \in S_{x,v} \\
nf(x),& \text{otherwise.}\end{cases}\] 
By Lemma \ref{lem:gen-simplex} 
\[\lim_{n\to \infty} (Kn)^{-\frac{1}{d}} \ell\left(\ppp{}\cap S_{x,v}\right) = \frac{c_d}{d(v_1\cdots v_d)^\frac{1}{d}}\]
with probability one,  and hence
\[\lim_{n\to \infty} n^{-\frac{1}{d}} \ell\left(\ppp{}\cap S_{x,v}\right) = \frac{c_d}{d}\left(\frac{\sup_{S_{x,v}} f}{v_1\cdots v_d}\right)^\frac{1}{d}\]
with probability one. Since $\ppp{nf} \subset \ppp{}$, $\ell\left( \ppp{nf}\cap S_{x,v}\right) \leq \ell\left( \ppp{}\cap S_{x,v}\right)$, which establishes \eqref{eq:gen-simplex-fmax}.

To prove \eqref{eq:gen-simplex-fmin}, note first that if $\inf_{S_{x,v}}  f =  0$, then the result is trivial. Hence, we may take $m:=\inf_{S_{x,v}} f > 0$.  For each point $X$ from $\ppp{nf}$ that falls in $S_{x,v}$, color $X$ red with probability $m/f(X)$. The red points form another Poisson point process on $S_{x,v}$ with intensity $mn$~\cite{kingman1992poisson}. Let us denote this Poisson point process by $\ppp{}^r$.  By Lemma \ref{lem:gen-simplex}, we have that 
\begin{equation}\label{eq:markedlim}
\limn (mn)^{-\frac{1}{d}} \ell\left( \ppp{}^r\cap S_{x,v}\right)  = \frac{c_d}{d (v_1\cdots v_d)^\frac{1}{d}} \ \ \text{with probability one}.
\end{equation}
Since $\ppp{}^r\subset\ppp{nf}$, $\ell(\ppp{}^r\cap S_{x,v}) \leq \ell(\ppp{nf}\cap S_{x,v})$. Combining this with \eqref{eq:markedlim} completes the proof.
\end{proof}

\begin{lemma}\label{lem:simplex-f-R}
Let $f\in \B$. Then with probability one
\begin{equation}\label{eq:fstarmax}
\limsupn n^{-\frac{1}{d}} \ell\left(\ppp{nf}\cap S_{x,v} \right) \leq\frac{c_d}{d} \left(\frac{\sup_{S_{x,v}} f^*}{v_1\cdots v_d}\right)^\frac{1}{d},
\end{equation}
and 
\begin{equation}\label{eq:fstarmin}
\liminfn n^{-\frac{1}{d}} \ell\left(\ppp{nf}\cap S_{x,v} \right) \geq\frac{c_d}{d} \left(\frac{\inf_{S_{x,v}} f_*}{v_1\cdots v_d}\right)^\frac{1}{d},
\end{equation}
for all $x,v \in (0, \infty)^d$.
\end{lemma}
\begin{proof}
Let us prove \eqref{eq:fstarmax}; the proof of \eqref{eq:fstarmin} is similar. For each $x,v \in (0,\infty)^d$ let $\Omega_{x,v}$ denote the event that \eqref{eq:gen-simplex-fmax} holds for $x,v$. By Lemma \ref{lem:simplex-f}, $\P(\Omega_{x,v})=1$. Let us set
\[\Omega = \bigcap \Big\{ \Omega_{x,v} \, : \, x,v \in \Q^d\cap(0,\infty)^d\Big\}.\]
Being the countable intersection of probability one events, $\Omega$ has probability one. 

Let $x,v \in (0,\infty)^d$ and let $\omega \in \Omega$. Let $q \in (0,\infty)^d\cap \Q^d$ such that $q_i < v_i$ for all $i$, and let $y \in (0,\infty)^d\cap \Q^d$ such that $y \geqq x$ and
\begin{equation}\label{eq:subset-req}
\frac{1}{q_j} \geq \frac{1+\langle y-x,v\rangle}{v_j}.
\end{equation}
We claim that $S_{x,v} \subset S_{y,q}$. To see this, let $z \in S_{x,v}$ and write
\[\langle y-z,q\rangle \stackrel{\eqref{eq:subset-req}}{\leq} \frac{\langle y-z,v\rangle}{1 + \langle y-x,v\rangle} = \frac{\langle x-z,v\rangle + \langle y-x,v\rangle }{1 + \langle y-x,v\rangle} \leq 1,\]
where we used in the last step that $\langle x-z,v\rangle \leq 1$. Therefore $z \in S_{y,q}$. It follows that $\ell(\ppp{nf}\cap S_{x,v}) \leq \ell(\ppp{nf}\cap S_{y,q})$, and from \eqref{eq:gen-simplex-fmax} we deduce 
\begin{equation}\label{eq:limsupQ}
\limsupn n^{-\frac{1}{d}} \ell\left(\ppp{nf}^\omega\cap S_{x,v}\right) \leq \frac{c_d}{d} \left(\frac{\sup_{S_{y,q}} f}{q_1\cdots q_d}\right)^\frac{1}{d}.
\end{equation}
Sending $y \to x$ in such a way that $y\geqq x$ and $y \in \Q^d\cap(0,\infty)^d$ yields
\[\limsupn n^{-\frac{1}{d}} \ell\left(\ppp{nf}^\omega\cap S_{x,v}\right) \leq \frac{c_d}{d} \left(\frac{\sup_{S_{x,q}} f^*}{q_1\cdots q_d}\right)^\frac{1}{d}.\]
The result follows by noting that the limit above holds for all $q \in (0,\infty)^d\cap \Q^d$ with $q_i <  v_i$ for all $i$.
\end{proof}

We denote by $B(x,r)$ the open ball of radius $r>0$ centered at $x \in \R^d$. We now prove our main consistency result.
\begin{theorem}[Consistency]\label{thm:con}
Let $f \in \B$. Then with probability one
\begin{equation}\label{eq:const1}
\limsup_{\eps \to 0^+}\limsup_{n\to \infty} \eps^{-1} n^{-\frac{1}{d}} \ell( \ppp{nf}\cap A_\eps(x_0)) \leq \frac{c_d}{d} \left(\frac{f^*(x_0)}{\phi_{x_1}(x_0)\cdots \phi_{x_d}(x_0)}\right)^\frac{1}{d},
\end{equation}
and
\begin{equation}\label{eq:const2}
\liminf_{\eps \to 0^+}\liminf_{n\to \infty} \eps^{-1}n^{-\frac{1}{d}} \ell( \ppp{nf}\cap A_\eps(x_0) \geq \frac{c_d}{d} \left(\frac{f_*(x_0)}{\phi_{x_1}(x_0)\cdots \phi_{x_d}(x_0)}\right)^\frac{1}{d},
\end{equation}
for all $x_0 \in (0,\infty)^d$ and all $\phi \in C^2(\R^d)$ such that $\phi_{x_i}(x_0) > 0$ for all $i$, where
\begin{equation}\label{eq:Aeps}
A_\eps(x_0) := \Big\{ x \in B(x_0,\sqrt{\eps}) \, : \, x \leqq x_0 \text{ and } \phi(x) \geq \phi(x_0) - \eps\Big\}.
\end{equation}
\end{theorem}
\begin{proof}
Let $\Omega$ be the event that \eqref{eq:fstarmax} and \eqref{eq:fstarmin} hold for all $x,v \in (0,\infty)^d$. By Lemma \ref{lem:simplex-f-R}, $\Omega$ has probability one. Let $x_0 \in (0,\infty)^d$, $\eps>0$ and let $\phi \in C^2(\Omega)$ such that $\phi_{x_i}(x_0) > 0$ for all $i$. Fix $\omega \in \Omega$.

We will give the proof of \eqref{eq:const1}; the proof of \eqref{eq:const2} is very similar. Since $\phi_{x_i}(x_0)>0$, there exists $m>0$ such that $\phi_{x_i}(x) \geq m$ for all $x \in A_\eps$, $i \in \{1,\dots,d\}$, and all $\eps>0$ sufficiently small. It follows that for any $x \in A_\eps$ 
\begin{equation}\label{eq:phi}
\phi(x_0) - \phi(x) = \int_0^1 \langle D\phi(x + t(x_0 - x)),x_0 - x\rangle \, dt \geq m \sum_{i=1}^d (x_{0,i} - x_i) \geq \frac{m|x_0 - x|}{\sqrt{d}}.
\end{equation}
We therefore deduce
\begin{equation}\label{eq:close}
|x_0 - x| \leq \frac{\sqrt{d}\eps}{m} \ \ \text{for all } x \in A_\eps.
\end{equation}
Since $\phi \in C^2$, there exists a constant $C>0$ such that
\begin{equation}\label{eq:phi2}
\phi(x) - \phi(x_0) \leq \langle D\phi(x_0),x-x_0\rangle  + C \eps^2
\end{equation} 
for all $x \in A_\eps$ and $\eps>0$ sufficiently small. It follows that
\[A_\eps(x_0) \subset \Big\{x \in \R^d \, : \, x \leqq x_0 \text{ and } \langle D\phi(x_0),x_0 - x\rangle \leq \eps + C\eps^2\Big\}.\]
Setting $q_\eps =  D\phi(x_0)/(\eps + C \eps^2)$ we have $A_\eps(x_0) \subset S_{x_0,q_\eps}$, and hence 
\[\ell(\ppp{nf}^\omega \cap A_\eps(x_0)) \leq \ell(\ppp{nf}^\omega\cap S_{x_0,q_\eps}).\]
 By  \eqref{eq:fstarmax}
\[\limsup_{n\to \infty} n^{-\frac{1}{d}} \ell( \ppp{nf}^\omega\cap A_\eps(x_0)) \leq \frac{c_d}{d} \left(\frac{\sup_{S_{x_0,q_\eps}} f^*}{\phi_{x_1}(x_0)\cdots \phi_{x_d}(x_0)}\right)^\frac{1}{d}(\eps + C\eps^2). \]
Sending $\eps\to 0^+$ completes the proof.
\end{proof}

\section{Proof of main result}
\label{sec:proof}

We now have the proof of Theorem \ref{thm:cont}.
\begin{proof}
Let $\Omega$ denote the event that the conclusions of Theorem \ref{thm:holder2} and Theorem \ref{thm:con} hold, and $U_n \equiv 0$ on $\partial \R^d_+$ for all $n$. Then $\P(\Omega)=1$. Let us fix a realization $\omega \in \Omega$. The remainder of the proof is split into several steps.

1. We first use stability (Theorem \ref{thm:holder2}) to obtain a compactness result.  Since $U^\omega_n(0)=0$ for all $n$, it follows from  Theorem \ref{thm:holder2} that the sequence $\big\{n^{-\frac{1}{d}}U_{n}^\omega(x)\big\}_{n\in \N}$ is bounded for all $x \in [0,\infty)^d$. By a diagonal argument, there exists a subsequence $\{U_{n_k}^\omega\}_{k\in \N}$ such that for all $x\in [0,\infty)^d \cap \Q^d$, $\big\{n_k^{-\frac{1}{d}} U_{n_k}^\omega(x) \big\}_{k \in \N}$ is a convergent sequence, whose limit we denote by $U^\omega(x)$.  By Theorem  \ref{thm:holder2}
\begin{equation}\label{eq:vholder}
|U^\omega(x) - U^\omega(y)| = \lim_{n_k\to \infty} n_k^{-\frac{1}{d}}|U^\omega_{n_k}(x) - U^\omega_{n_k}(y)| \leq c_ddM^\frac{d-1}{d}\|f^*\|_{L^\infty([0,M]^d)}^\frac{1}{d}|x-y|^\frac{1}{d}
\end{equation}
for all $x,y \in \Q^d\cap [0,\infty)^d$, where $M=\max\{x_1,y_1,x_2,y_2,\dots,x_d,y_d\}$.  Hence, we can extend $U^\omega$ uniquely to a function $U^\omega \in C([0,\infty)^d)$ such that for every $M>0$, $U^\omega \in C^{0,\frac{1}{d}}([0,M]^d)$ and
  \begin{equation}\label{eq:holder-norm}
[U^\omega]_{\frac{1}{d};[0,M]^d} \leq c_d d M^\frac{d-1}{d} \|f^*\|_{L^\infty([0,M]^d)}^\frac{1}{d}.
  \end{equation}
Furthermore, $U^\omega$ is nondecreasing and $U^\omega \equiv 0$ on $\partial \R^d_+$.

We claim that $n_k^{-\frac{1}{d}} U_{n_k}^\omega \longrightarrow U^\omega$ locally uniformly on $[0,\infty)^d$. The proof of this is similar to \cite[Theorem 1]{calder2014}. We include it here for completeness. Fix $M>0$ and let $\eps>0$. Let $T \in \N$ and for any multi-index $\alpha \in \Z^d$, let $x_\alpha = \alpha/T$. Let $I$ be the set of multi-indices $\alpha \in \Z^d$ for which $x_\alpha \in [0,M]^d$. Since $U^\omega$ is continuous on $[0,M]^d$, we can choose $T$ large enough so that
\begin{equation}\label{eq:grid-cont}
U^\omega(x_{\alpha + \vb{1}}) - U^\omega(x_\alpha) < \eps \ \ \text{ for all } \alpha \in I.
\end{equation}
Since $I$ is a finite set and $x_\alpha \in \Q^d$ for all $\alpha \in I$, we deduce
\begin{equation}\label{eq:grid-lim}
\lim_{k\to\infty} \max_{\alpha \in I} \left| n_k^{-\frac{1}{d}} U_{n_k}^\omega(x_\alpha) - U^\omega(x_\alpha) \right| = 0.
\end{equation}
Let $y \in [0,M]^d$ and let $\alpha \in I$ such that $x_\alpha \leqq y \leqq x_{\alpha+\vb{1}}$. Since $U^\omega_n$ and $U^\omega$ are nondecreasing  we have
\[n^{-\frac{1}{d}} U^\omega_n(y) - U^\omega(y) \leq n^{-\frac{1}{d}} U^\omega_n(x_{\alpha+\vb{1}}) - U^\omega(x_\alpha) \stackrel{\eqref{eq:grid-cont}}{<} n^{-\frac{1}{d}} U^\omega_n(x_{\alpha+\vb{1}}) - U^\omega(x_{\alpha+\vb{1}}) + \eps.\]
Similarly, we deduce
\[n^{-\frac{1}{d}} U^\omega_n(y) - U^\omega(y) \geq n^{-\frac{1}{d}} U^\omega_n(x_{\alpha}) - U^\omega(x_{\alpha+\vb{1}}) \stackrel{\eqref{eq:grid-cont}}{>} n^{-\frac{1}{d}} U^\omega_n(x_{\alpha}) - U^\omega(x_{\alpha}) - \eps.\]
Combining these inequalities yields
\begin{equation}\label{eq:Linfty-bound}
\|n^{-\frac{1}{d}} U^\omega_n - U^\omega \|_{L^\infty([0,M]^d)} < \max_{\alpha \in I} \left| n^{-\frac{1}{d}} U_{n}(x_\alpha) - U^\omega(x_\alpha) \right| + \eps.
\end{equation}
Invoking \eqref{eq:grid-lim} we see that
\[ \limsup_{k\to\infty} \|n_k^{-\frac{1}{d}} U_{n_k}^\omega - U^\omega \|_{L^\infty([0,M]^d)}  < \eps.\]
Sending $\eps\to0$ establishes the claim.

2. We now show that $U^\omega$ is a viscosity subsolution of (P). For simplicity, let us set 
\[V_k= n_k^{-\frac{1}{d}}U_{n_k}^\omega \ \ \text{ and } \ \ V = U^\omega.\] 
Fix $M>0$ and let $x_0 \in (0,M)^d$. Let $\phi\in C^2(\R^d)$ such that $V - \phi$ has a local maximum at $x_0$. Since $V$ is nondecreasing, $\phi_{x_i}(x_0) \geq 0$ for all $i$. If $\phi_{x_i}(x_0)=0$ for some $i$, then the subsolution property is trivially satisfied. Hence, we may assume that $\phi_{x_i}(x_0) > 0$ for all $i$. Without loss of generality, we may also assume that $V - \phi$ has a strict maximum at $x_0$, relative to the set $[0,M]^d$. Then there exists a sequence $\{x_k\}_{k\in \N}$ in $[0,M]^d$ converging to $x_0$ such that $V_k - \phi$ has a maximum at $x_k$, relative to $[0,M]^d$. 

Let $r,m>0$ such that $B(x_0,r) \subset (0,M)^d$ and $\phi_{x_i}(x) > m$ for all $x \in B(x_0,r)$ and all $i$. Let $\eps\in (0,1)$ such that
\begin{equation}\label{eq:epscond}
\eps < \min\left\{\frac{m^2}{4d},r^2\right\}.
\end{equation}
Since $V_k \to V$ uniformly on $[0,M]^d$ and $x_k \to x_0$, there exists $K>0$ such that for all $k>K$
\begin{equation}\label{eq:kcond}
\phi(x_0) - \phi(x_k) + V_k(x_k) - V_k(x_0) < \eps^2.
\end{equation}
Let
\begin{equation}\label{eq:Ake}
\A_{k,\eps}=\big\{ z \in [0,x_0] \, : \, V_k(z) \geq V_k(x_0) - \eps\big\}.
\end{equation}

We claim that $\A_{k,\eps} \subset B(x_0,\sqrt{\eps})$ for all $k>K$. To see this, let $z \in B(x_0,r)$ such that $z \leqq x_0$ and note that
\begin{equation}\label{eq:dec}
\phi(x_0) - \phi(z) = \int_0^1 \left\langle D\phi(z + t(x_0-z)),x_0-z\right\rangle\, dt \geq m\sum_{i=1}^d (x_{0,i} - z_i) \geq \frac{m|x_0-z|}{\sqrt{d}}.
\end{equation}
Since $V_k - \phi $ has a maximum at $x_k$,
\begin{equation}\label{eq:Vkmax}
V_k(z) - \phi(z) \leq V_k(x_k) - \phi(x_k) \ \ \text{ for } z \in [0,M]^d.
\end{equation}
Combining this with \eqref{eq:epscond}, \eqref{eq:kcond} and \eqref{eq:dec} we deduce
\begin{align*}
V_k(z) &\leq V_k(x_k) + \phi(z) - \phi(x_k) \\
&= V_k(x_0) + \phi(z) - \phi(x_0) + \phi(x_0) - \phi(x_k)+ V_k(x_k) - V_k(x_0)\\
&< V_k(x_0)  - \frac{m}{\sqrt{d}}\sqrt{\eps} + \eps^2\\
&< V_k(x_0) - \eps
\end{align*}
for all $z \in \partial B(x_0,\sqrt{\eps})$ with $z \leqq x_0$, and $k>K$. Since $V_k$ is nondecreasing, $\A_{k,\eps} \subset B(x_0,\sqrt{\eps})$, establishing the claim.

We now claim that
\begin{equation}\label{eq:Ake-chain}
\eps n_k^\frac{1}{d} \leq \ell(\ppp{n_kf}^\omega \cap \A_{k,\eps}).
\end{equation}
To see this, since $U_{n}$ is integer valued, we can write $\A_{k,\eps}$ as
\begin{equation}\label{eq:Ake-int}
\A_{k,\eps}=\big\{ z \in [0,x_0] \, : \, U_{n_k}^\omega(z) \geq U_{n_k}^\omega (x_0) - \lfloor \eps n_k^\frac{1}{d}\rfloor\big\},
\end{equation}
where $\lfloor t\rfloor$ is the largest integer less than or equal to $t \in \R$.
Since $\A_{k,\eps} \subset B(x_0,r)\subset (0,\infty)^d$ and $U^\omega_{n_k}(y) = 0$ for $y \in \partial \R^d_+$, we must have that $U^\omega_{n_k}(x_0) - \lfloor \eps n_k^\frac{1}{d}\rfloor\geq 1$. Let $L = U_{n_k}^\omega(x_0)$ and let $X^\omega_1\leqq \dots\leqq X^\omega_{L}$ be a chain in $\ppp{n_kf}^\omega\cap [0,x_0]$ of length $L$. For any $q \geq L-\lfloor \eps n_k^\frac{1}{d}\rfloor$, $U_{n_k}^\omega(X^\omega_q) \geq q \geq  L - \lfloor \eps n_k^\frac{1}{d}\rfloor$, and hence $X^\omega_q \in \A_{k,\eps}$. Hence the chain $X^\omega_q, \dots,X^\omega_L$ belongs to $\A_{k,\eps}$ for $q=L-\lfloor \eps n_k^\frac{1}{d} \rfloor$ and
\[\eps n_k^\frac{1}{d} \leq \lfloor \eps n_k^\frac{1}{d}\rfloor + 1 = L-q + 1 \leq \ell(\ppp{n_kf}^\omega\cap \A_{k,\eps}),\]
which establishes the claim.

By \eqref{eq:kcond} and \eqref{eq:Vkmax}
\[V_k(z) - V_k(x_0) \leq \phi(z) - \phi(x_0) + \eps^2\]
for all $k>K$ and $z \in [0,M]^d$. Since $\A_{k,\eps}\subset B(x_0,\sqrt{\eps})$ it follows that
\begin{equation}\label{eq:vksubphi}
\A_{k,\eps} \subset \big\{ z \in B(x_0,\sqrt{\eps}) \, : \, x \leqq x_0 \ \text{ and } \ \phi(z) \geq \phi(x_0) -\eps-\eps^2)\big\}=:A_{\eps+\eps^2}(x_0)
\end{equation}
for $k>K$, where $A_{\eps+\eps^2}(x_0)$ is defined as in Theorem \ref{thm:con}. Invoking \eqref{eq:Ake-chain} and the monotonicity of $\ell$ we have
\[\eps \leq\limsup_{k\to \infty}n_k^{-\frac{1}{d}}\ell(\ppp{n_kf}^\omega\cap \A_{k,\eps}) \leq \limsup_{n\to \infty} n^{-\frac{1}{d}}\ell\left( \ppp{n f}^\omega \cap A_{\eps+\eps^2}(x_0)\right).\]
Since Theorem \ref{thm:con} holds for all $\omega \in \Omega$ and $\phi_{x_i}(x_0)>0$, we deduce
\[1 \leq \limsup_{\eps \to 0}\limsup_{n\to \infty} (\eps+\eps^2)^{-1} n^{-\frac{1}{d}}\ell\left( \ppp{n f}^\omega \cap A_{\eps+\eps^2}(x_0)\right) \leq \frac{c_d}{d} \left(\frac{f^*(x_0)}{\phi_{x_1}(x_0)\cdots \phi_{x_d}(x_0)}\right)^\frac{1}{d},\]
and hence
\[\phi_{x_1}(x_0) \cdots \phi_{x_d}(x_0) \leq \frac{c_d^d}{d^d}f^*(x_0).\]

3. We now show that $V$ is a viscosity supersolution of (P). Let $\phi\in C^2(\R^d)$ such that $V - \phi$ has a local minimum at $x_0$. Without loss of generality, we may assume that $x_0$ is a minimum of $V-\phi$ relative to the set $[0,M]^d$. Since $V$ is nondecreasing, $\phi_{x_i}(x_0) \geq 0$ for all $i$.  For $\lambda>0$, set
\[\phi^\lambda(x) = \phi(x) + \lambda(x_1+\cdots +x_d).\]
Then $V-\phi^\lambda$ has a strict minimum at $x_0$, relative to the set $[0,x_0]$, and $\phi^\lambda_{x_i}(x_0)>0$ for all $i$. Therefore, there exists a sequence $\{x_k\}_{k\in \N}$ in $[0,x_0]$ converging to $x_0$ such that $V_k - \phi^\lambda$ has a minimum at $x_k$, relative to $[0,x_0]$.  

Let $r,m>0$ such that $B(x_0,r) \subset (0,M)^d$ and $\phi^\lambda_{x_i}(x) > m$  for all $x \in B(x_0,r)$ and all $i$.
Let $\eps>0$ and let $X^\omega_1,\dots,X^\omega_j$ be a chain in $\ppp{n_k f}^\omega \cap \A_{k,\eps}$. Then by the definition of $\A_{k,\eps}$ \eqref{eq:Ake-int}
\[U^\omega_{n_k}(X_1^\omega) \geq U_{n_k}^\omega(x_0) - \lfloor \eps n_k^\frac{1}{d}\rfloor,\]
and therefore
\[U_{n_k}^\omega(x_0) \geq U_{n_k}^\omega(X_j^\omega) \geq U_{n_k}^\omega(X_1^\omega) + j-1 \geq U^\omega_{n_k}(x_0) - \lfloor \eps n_k^{\frac{1}{d}}\rfloor + j-1.\]
Hence $j \leq  \eps n_k^\frac{1}{d} + 1$ and  therefore
\begin{equation}\label{eq:LAe}
\ell(\ppp{n_kf}^\omega(\A_{k,\eps})) \leq \eps n_k^\frac{1}{d} + 1.
\end{equation}

Since $V_k - \phi^\lambda$ has a minimum at $x_k$
\[V_k(z) - V_k(x_0) \geq \phi^\lambda(z) - \phi^\lambda(x_0) + \phi^\lambda(x_0) - \phi^\lambda(x_k) + V_k(x_k) - V_k(x_0),\]
for $z \in [0,x_0]$. Since $x_k \to x_0$ and $V_k \to V$ locally uniformly, we can choose $k$ larger, if necessary, so that
\[V_k(z) - V_k(x_0) \geq \phi^\lambda(z) - \phi^\lambda(x_0) - \eps^2 \ \ \text{ for all } z \in [0,x_0].\]
It follows that
\[\A_{k,\eps} \supset \{z \in B(x_0,\sqrt{\eps})\, : \, x \leqq x_0 \ \text{and} \  \phi^\lambda(z) \geq \phi^\lambda(x_0) + \eps^2 - \eps\}=:A_{\eps-\eps^2}(x_0),\]
where $A_{\eps-\eps^2}(x_0)$ is as defined in Theorem \ref{thm:con}. Invoking \eqref{eq:LAe} and the monotonicity of $\ell$
\[\eps n_k^\frac{1}{d} + 1 \geq \ell\left( \ppp{n_k f}^\omega \cap \A_{k,\eps} \right) \geq \ell(\ppp{n_k f}^\omega \cap A_{\eps - \eps^2}(x_0) ),\]
and therefore
\[\eps \geq \liminf_{n\to \infty} n^{-\frac{1}{d}} \ell(\ppp{n f}^\omega \cap A_{\eps - \eps^2}(x_0) ).\]
Since Theorem \ref{thm:con} holds for all $\omega \in \Omega$ and $\phi^\lambda_{x_i}(x_0)>0$
\[1 \geq \liminf_{\eps\to 0}\liminf_{n \to \infty}(\eps-\eps^2)^{-1}n^{-\frac{1}{d}}\ell(\ppp{n f}^\omega \cap A_{\eps-\eps^2}(x_0)) \geq \frac{c_d}{d} \left(\frac{f_*(x_0)}{\phi^\lambda_{x_1}(x_0) \cdots \phi^\lambda_{x_d}(x_0)}\right)^\frac{1}{d}.\]
Thus, we arrive at
\[(\phi_{x_1}(x_0)+\lambda) \cdots (\phi_{x_d}(x_0)+\lambda) \geq \frac{c_d^d}{d^d} f_*(x_0).\]
Since $\lambda>0$ was arbitrary, $V$ is a viscosity supersolution of (P).

4. By uniqueness of nondecreasing viscosity solutions of (P), we have $U^\omega= u$.  Since we can apply the same argument to any subsequence of $\{n^{-\frac{1}{d}}U^\omega_n\}_n$, and extract a further subsequence converging locally uniformly to $u$, we find that $n^{-\frac{1}{d}}U^\omega_n \longrightarrow u$ locally uniformly on $[0,\infty)^d$ for all $\omega \in \Omega$, where $\Omega$ has probability one. 
\end{proof}


\appendix
\section{Complete convergence for longest chain problem}
We sketch the proof of Theorem \ref{thm:complete}.
\begin{proof}
Let $\eps>0$. Let $X_1,X_2,X_3,\dots$ be a sequence of independent and uniformly distributed random variables on $[0,1]^d$ and let $\ell_n = \ell(\{X_1,\dots,X_n\})$ be the length of a longest chain. Let $N$ denote the cardinality of $\ppp{t}\cap[0,1]^d$. Conditioned on $N=n$, $L_t$ and $\ell_n$ have the same distribution, so by the usual tail bounds on Poisson random variables we deduce
\begin{equation}\label{eq:poissonconcentration}
P(L_t \geq \lambda)\leq P\left(\ell_{\lfloor t + t^\frac{3}{4}\rfloor}\geq \lambda\right) + \exp\left(\frac{-t^\frac{1}{2}}{2}\right).
\end{equation}
By \cite{bollobas1988} there exist constants $0<C_1<C_2$ such that
\begin{equation}\label{eq:bounds}
C_1n^\frac{1}{d}\leq  \E[\ell_n] \leq C_2n^\frac{1}{d}
\end{equation}
for all $n\geq 1$, and 
\begin{equation}\label{eq:mean}
\lim_{n\to \infty} n^{-\frac{1}{d}} \E[\ell_n] = c_d>0.
\end{equation}
Since the longest chain function $\ell$ is a \emph{configuration function} (see \cite[Definition 7.1.7]{talagrand1995concentration}), we have
\begin{equation}\label{eq:concentration}
P\left(\ell_n \leq \M[\ell_n]-\lambda\right),P\left(\ell_n \geq \M[\ell_n]+\lambda\right) \leq 2\exp\left(-\frac{\lambda^2}{4(\M[\ell_n] + \lambda)} \right),
\end{equation}
for any $\lambda>0$, where $\M(\ell_n)$ denotes any median of $\ell_n$. A short computation involving \eqref{eq:bounds} and  \eqref{eq:concentration} yields
\begin{equation}\label{eq:med}
|\E[\ell_n]  - \M(\ell_n)| \leq C\log(n)n^\frac{1}{2d}.
\end{equation}
and we find that $n^{-\frac{1}{d}}\M[\ell_n] \to c_d$ as $n\to \infty$. Fix $T$ large enough so that 
\begin{equation}\label{eq:medeps}
|\M[\ell_n]  - c_dn^\frac{1}{d}| \leq \frac{\eps}{2} n^\frac{1}{d} \ \ \text{ for all } n > T.
\end{equation}
It is well known (see, e.g., \cite{steele1997probability}) that
\[|\E[\ell_n] - \E[\ell_k]| \leq C|n-k|^\frac{1}{d}.\]
Combining this with \eqref{eq:med}  we deduce
\begin{equation}\label{eq:med_smooth}
|\M[\ell_n] - \M[\ell_k]| \leq C\left(\log(k)k^\frac{1}{2d} + \log(n) n^\frac{1}{2d} + |n-k|^\frac{1}{d} \right).
\end{equation} 
Recalling \eqref{eq:poissonconcentration}, \eqref{eq:concentration} and \eqref{eq:medeps} we have 
\begin{equation}\label{eq:Ltcon}
P(L_t \geq c_dt^\frac{1}{d} + \eps t^\frac{1}{d})\leq 2\exp\left(-C\eps^2 t^\frac{1}{d}\right)+  \exp\left(\frac{-t^\frac{1}{2}}{2}\right).
\end{equation}
for all $t>T$ sufficiently large. The other inequality is similar, and the result follows from \eqref{eq:seqreq}.
\end{proof}

\section*{Acknowledgments}
The author is grateful to Lawrence C.~Evans and Fraydoun Rezakhanlou for valuable discussions about this work.

\end{document}